\documentclass[12pt,twoside,a4paper]{amsart}

\usepackage[margin=2.5cm]{geometry}

\usepackage{amssymb,mathtools}
\usepackage{xcolor}
\usepackage{hyperref}

\newtheorem{thm}{Theorem}[section]
\newtheorem{lem}[thm]{Lemma}
\newtheorem{prop}[thm]{Proposition}
\newtheorem{cor}[thm]{Corollary}
\theoremstyle{definition}
\newtheorem{defn}[thm]{Definition}
\theoremstyle{remark}

\numberwithin{equation}{section}
\allowdisplaybreaks

\newcommand{\R}{\mathbb{R}}

\newcommand{\HD}{\dim_H}
\newcommand{\Hau}{\mathcal{H}}
\newcommand{\IP}{\operatorname{IP}}

\begin{document}
\keywords{Hausdorff measure, projection theory, self-similar sets}
\subjclass[2020]{Primary 28A80; Secondary 28A75 28A78}
\thanks{The author would like to thank Pertti Mattila and Chun-Kai Tseng for  reading an early version of the manuscript. The author is supported by NSTC through grant 111-2115-M-002-010-MY5. 
}

\title[Hausdorff measure and self-similar sets]{On the Hausdorff measure of projections of self-similar sets}

\author{Chong-Wei Liang}
\address{Department of Mathematics, National Taiwan University, Taiwan.}
\email{d10221001@ntu.edu.tw}


\begin{abstract}
For $0<d\leq1/4$, let $\mathcal{C}(d)$ be the four-corner Cantor set with Hausdorff dimension $s_d=\log 4/\log(1/d)$. Peres, Simon, and Solomyak, as well as Mattila, asked for which $d$ 
the measure $\Hau^{s_d}(p_\theta(\mathcal{C}(d)))$ is positive for almost every direction $\theta$. It was open for the range $1/9\leq d\leq1/6$. In this paper, we show that, for $\delta< d\leq1/6$, there is a set $\IP(d)$ of positive Lebesgue measure such that for almost every $\theta\in\IP(d)$,  $\Hau^{s_d}(p_\theta(\mathcal{C}(d)))=0$, where $\delta=0.155124983896014\ldots$ is the unique zero in
$(1/9,1/6)$ of the polynomial
$P(d)=1-7d+3d^2+4d^3-2d^4-d^5$.
\end{abstract}
\maketitle

\section{Introduction and the main results}
For $0<d\leq 1/4$, let
\begin{equation}\label{eq:Kd}
  K_d=(1-d)\left\{\sum_{n=0}^{\infty}a_nd^n:
  a_n\in\{0,1\}\right\}\quad\text{and}\quad \mathcal{C}(d)=K_d\times K_d.
\end{equation}
 The set $\mathcal{C}(d)$ is the four-corner Cantor set obtained by copying each piece and scaling it by a factor of $d$, at the corners of the unit square.  Its
Hausdorff dimension is $  s_d={\log 4}/{\log(1/d)}$
and $0<\Hau^{s_d}(\mathcal{C}(d))<\infty$.

For $\theta\in[0,\pi)$, let $p_\theta$ be the orthogonal projection given by
\[p_\theta(x,y)=x\cos\theta+y\sin\theta.
\]
Marstrand's projection theory \cite{MR0063439} shows that $\HD p_\theta(\mathcal{C}(d))=s_d$ for almost all $\theta$. A natural question raised by Peres, Simon, and Solomyak \cite[Example 6.1]{MR1760599} (see also Mattila \cite[Problem~1, page.~39]{MR2044636}), asks for which
$d$ one has
\begin{equation}\label{eq:Mattila-question}
  \Hau^{s_d}(p_\theta(\mathcal{C}(d)))>0
  \quad\text{for Lebesgue almost every }\theta\in[0,\pi).
\end{equation}
When $d =1/4$, $\mathcal{C}(1/4)$ is an example of a purely unrectifiable 
set with finite and positive $\Hau^1$ measure. By a general theorem of Besicovitch \cite{MR0867284,MR1333890},  we have that $\Hau^1(p_\theta(C(1/4)))=0$ for almost all $\theta$.  Peres, Simon, and Solomyak
proved that (\ref{eq:Mattila-question}) fails for $1/6<d<1/4$ \cite[Theorem 1.2 and Example 6.1]{MR1760599}.
The positive conclusion was known for $d<1/9$ \cite[Section~1]{MR2044636}. 
The range $1/9\leq d\leq 1/6$ was left open.

We make partial progress on this question in this paper. We prove that there is $\delta\in(1/9,1/6)$ such that (\ref{eq:Mattila-question}) is false for all $\delta< d\leq1/6$.
The threshold is the number $\delta=0.1551249838960142617871334092\ldots$, which is
the unique zero in $(1/9,1/6)$ of the polynomial
\begin{equation}\label{eq:P}
  P(d)=1-7d+3d^2+4d^3-2d^4-d^5.
\end{equation}
The choice of the threshold $\delta$ depends on the choice of the covering intervals (see Lemma \ref{lem:parameter}).

Define the set
\begin{equation}\label{eq:Fd}
  F_d=K_d-K_d
  =(1-d)\left\{\sum_{n=0}^{\infty}\varepsilon_nd^n:
  \varepsilon_n\in\{-1,0,1\}\right\}.
\end{equation}
For a set $F\subset\R$, we write $ F/F$ to represent the set of quotient $x/y$, in which $x,y\in F$ with $y\neq0$. The new part of Theorem \ref{mainthm} is
$\delta<d\leq1/6$. 

\begin{thm}\label{mainthm}
For every $\delta<d<1/4$, the quotient set $F_d/F_d$ contains a
nondegenerate interval.  At $d=1/6$ one has
\begin{equation}\label{eq:explicit-sixth}
  [10/7,35/24]\subset F_{1/6}/F_{1/6}.
\end{equation}
\end{thm}

The connection between $F_d/F_d$ and the problem is based on theorem of Peres, Simon, and Solomyak \cite{MR1760599}. The existence of a nondegenerate  interval in this quotient
gives a positive-measure interval of directions on which $p_\theta$ is not
one-to-one on $\mathcal{C}(d)$.  The projection theorem of Peres, Simon, and Solomyak gives zero critical Hausdorff measure for almost every direction in
that interval.  We illustrate this reduction in Section \ref{sec:2}.

As a corollary, we show (\ref{eq:Mattila-question}) is false for all $\delta< d\leq1/6$.
\begin{cor}\label{Mattila}
For $\delta< d\leq1/6$, there is a subset of direction $\IP(d)\subset[0,\pi)$ of positive Lebesgue measure such that for almost every $\theta\in\IP(d)$,  $\Hau^{s_d}(p_\theta(\mathcal{C}(d)))=0$.
\end{cor}

\noindent{\it Organization.} The paper is organized as follows. In Section \ref{sec:2}, the review background material and the theorem of Peres, Simon, and Solomyak are included. Section \ref{sec:partI} and \ref{sec:partII} contribute to the verification of Theorem \ref{mainthm} and Corollary \ref{Mattila}.

\section{Preliminaries and Theorem of Peres, Simon, and Solomyak}\label{sec:2}
\subsection{Open set condition and self-similar sets in the plane}
We recall some terminology associated with the self-similar sets.
\begin{defn}[IFS]\label{def:ifs}
    Fix an interval $I\subset\R$.
    Define the {\it iterated function system} ({\it IFS} ) $\{S_i\}$ to be a finite family of contractive maps $S_i:I\rightarrow I$ with contraction factor $0<d_i<1$.
    The {\it attractor} of $\{S_i\}$ is the unique nonempty compact set $\mathcal{C}\subset I$ satisfying 
$$\mathcal{C}=\bigcup_{i}S_i(\mathcal{C}).$$
    We say an IFS $\{S_i\}$ has {\it open set condition} ({\it OSC}) if there is a nonempty open set $U$ such that $S_i(U)$ are disjoint and lie in $U$ for all $i$. 
\end{defn}
Consider a self-similar set $\mathcal{C}\subset\R^2$, defined as the unique
non-empty compact satisfying
\begin{align}\label{selfsimilar}
\mathcal{C}=\bigcup_i(d_i\mathcal{C}+b_i)\quad\text{with}\,\,d_i\in(0,1)\quad\text{and}\quad b_i\in\R^2.
\end{align}
 We \emph{always assume} that the similitudes $S_i(x):=d_ix+b_i$ for all $i$ satisfy the {\it OSC}. It is well-known \cite{MR0625600} that the Hausdorff dimension $\HD\mathcal{C}$ equals the similarity dimension $s$, defined by $\sum_i d^s_i=1$,
and the $s$-dimensional Hausdorff measure of $\mathcal{C}$ is positive and finite.

\subsection{Theorem of Peres, Simon, and Solomyak and its implication}
Let $\IP(\mathcal{C})$ be the set of directions $\theta\in[0,\pi)$ for which
$p_\theta$ is not one-to-one on $\mathcal{C}$ (the letters "$\IP$" stand for "intersection parameters").
We will use the following theorem stated in
\cite[Theorem~1.2]{MR1760599}.  Ero\u{g}lu proved the corresponding
statement for products of affine Cantor sets
\cite[Theorem 2.14]{MR2312386}.

\begin{thm}\label{thm:PSS}
Let $\mathcal{C}\subset\R^2$ be a self-similar set (\ref{selfsimilar}) with Hausdorff dimension $0<s<1$.  Suppose that $\mathcal{C}$ is not contained in a line.  Then
\[
  \Hau^s(p_\theta(\mathcal{C}))=0
  \quad\text{for almost every }\theta\in\IP(\mathcal{C}).
\]
\end{thm}

We will apply this theorem to the four-corner Cantor set $\mathcal{C}(d)=K_d\times K_d$, where $0<d<1/4$. The main observation is the  Proposition \ref{prop:criterion}.\\ Let $\IP(d):=\IP(\mathcal{C}(d))$. Then $\theta\in\IP(d)$ if
there is a nonzero vector $z\in \mathcal{C}(d)-\mathcal{C}(d)=F_d\times F_d$ such that
$p_\theta(z)=0$.

\begin{prop}\label{lem:quotient-to-IP}
Let $I\subset(0,\infty)$ be a nondegenerate interval. If the interval $I$ is contained in $F_d/F_d$,
then $\arctan I\subset\IP(d)$.  In particular, $\IP(d)$ has positive
Lebesgue measure.
\end{prop}
\begin{proof}
    Let $t\in I$. There exists $a,b\in F_d$ with $b>0$ such that $t=a/b=(x_1-x_2)/(y_1-y_2)$ for some $x_1,x_2, y_1,y_2\in K_d$. The difference of the two points $(x_1,y_2), (x_2,y_1)$ is $(a,-b)\in F_d\times F_d$. If $\theta=\arctan t$, then $p_\theta(a,-b)=0$ and hence $\arctan I\subset\IP(d)$. The map $t\mapsto\arctan t$ is a $C^1$
diffeomorphism on $I$, so $\IP(d)$ has positive Lebesgue measure.
\end{proof}
 
Our criterion for verifying $ \Hau^{s_d}(p_\theta(\mathcal{C}(d)))=0$, combining with Theorem \ref{thm:PSS}, is to show $F_d/F_d$ has a nonempty interior.

\begin{prop}\label{prop:criterion}
Let $0<d< 1/4$. If $F_d/F_d$ contains a nondegenerate interval, then
there is a set of direction $\IP(d)\subset[0,\pi)$ of positive Lebesgue measure such that for almost every $\theta\in\IP(d)$,  $\Hau^{s_d}(p_\theta(\mathcal{C}(d)))=0$.
\end{prop}

\begin{proof}
The four-corner Cantor set $\mathcal{C}(d)\subset\R^2$ is the attractor of the IFS 
\begin{align*}
\left\{dx+(1-d)\varepsilon:\,\varepsilon\in\{0,1\}^2\right\},
\end{align*}satisfies the {\it OSC} with the open set $U=(0,1)^2$ and has similarity
dimension $s_d\in(0,1)$.  By Proposition \ref{lem:quotient-to-IP}, $\IP(d)$ has
positive Lebesgue measure.  Theorem \ref{thm:PSS} gives
$\Hau^{s_d}(p_\theta(\mathcal{C}(d)))=0$ for almost every $\theta\in\IP(d)$.
\end{proof}

\section{Proof of Theorem \ref{mainthm}--part \rm{I}}\label{sec:partI}

It is convenient to translate $F_d$ to a Cantor-like set with digits $0,1,2$. Define
\begin{equation}\label{eq:Gamma}
  \Gamma_d=1+F_d
  =(1-d)\left\{\sum_{n=0}^{\infty}\varepsilon_nd^n:
  \varepsilon_n\in\{0,1,2\}\right\}
\end{equation}
and, for $t>1$, let 
\begin{equation}\label{eq:R_t}
    R_t=\Gamma_d+t\Gamma_d=(1-d)\left\{\sum_{n=0}^{\infty}(i_n+tj_n)d^n:
  i_n,j_n\in\{0,1,2\}\right\}.
\end{equation} 
The set $\Gamma_d$ possesses the symmetry that $\Gamma_d=2-\Gamma_d$ and the fractal set $R_t$ is the attractor of the nine similarities
\begin{equation}\label{eq:phi}
  \phi_{ij}(x)=dx+(1-d)(i+tj),
  \qquad i,j\in\{0,1,2\}.
\end{equation}

We first record a backward covering observation.

\begin{lem}\label{lem:backward-cover}
Let $\mathcal K\subset\R$ be nonempty, compact, and bounded.  If $\mathcal K\subset
  \bigcup_{i,j=0}^2\phi_{ij}(\mathcal K)$, then
$\mathcal K\subset R_t$.
\end{lem}

\begin{proof}
Given $x_0\in\mathcal K$, choose recursively $x_n\in\mathcal K$ and
$i_n,j_n\in\{0,1,2\}$ such that
$x_{n-1}=\phi_{i_nj_n}(x_n)$.  Then
\[
  x_0=(1-d)\sum_{k=1}^n d^{k-1}(i_k+t j_k)+d^n x_n.
\]
The desired result follows from the assumption that $\mathcal{K}$ is compact and bounded and the limiting argument.
\end{proof}

\subsection{Finding the interval in the fractal set $R_t$}
We will apply Lemma \ref{lem:backward-cover} to find an interval $J$ depending on $t$ such that $J\subset R_t$.
Consider the  intervals
\begin{equation}\label{eq:findingintervals}
\begin{split}
  J&=[1,1+2t],\quad A=\phi_{00}(J),\quad B=\phi_{10}(A),\quad  D=\phi_{22}(J)\quad\text{and}\quad G=\phi_{12}(D).
\end{split}
\end{equation}
The first observation is that, under suitable assumption on $t$, the eleven intervals
\begin{equation}\label{eq:eleven-intervals}
\begin{split}
 &\phi_{10}(J),\ \phi_{01}(J),\ \phi_{20}(B),\ \phi_{01}(G),
 \ \phi_{20}(J),\ \phi_{11}(J),\\
 &\phi_{02}(J),\ \phi_{21}(B),\ \phi_{02}(G),\ \phi_{21}(J),
 \ \phi_{12}(J).
\end{split}
\end{equation}  will cover $J$. Here and below the intervals are listed from left to right.

\begin{lem}\label{lem:findingcover}
Define
\begin{equation}\label{eq:LUV}
 L(d)=\frac{2-2d-d^2+d^3}{1+d},\quad U(d)=\frac{2(1-d-d^2+d^3)}{1+d-4d^3}\quad\text{and}\quad V(d)=\frac{1-d}{1-3d}.
\end{equation}
If $ L(d)\leq t\leq\min\{U(d),V(d)\}$, then 
 $J$ is covered by the intervals in \eqref{eq:eleven-intervals}. Besides,
\[
  J\cup A\cup B\cup D\cup G\subset R_t.
\]
\end{lem}

\begin{proof}
Remark that two consecutive intervals overlap if the right endpoint of the former interval is larger than the left endpoint of the lateral interval. Hence, by brutal force, each pair of consecutive intervals in (\ref{eq:eleven-intervals}) overlaps if 
\begin{align*}
    & 1-d-t+3dt\geq0,\\
 &(1+d)t-(2-2d-d^2+d^3)\geq0,\\
&2-2d-2d^2+2d^3-(1+d-4d^3)t\geq 0.
\end{align*}
These three conditions are equivalent to
$t\leq V(d)$, $t\geq L(d)$, and $t\leq U(d)$, respectively.  This proves
$J$ is covered by the intervals in \eqref{eq:eleven-intervals} when $ L(d)\leq t\leq\min\{U(d),V(d)\}$.

Let $\mathcal{K}:= J\cup A\cup B\cup D\cup G$. Then by the covering property of $J$ and Lemma \ref{lem:backward-cover} gives that $J\cup A\cup B\cup D\cup G\subset R_t$.
\end{proof}

\subsection{Finding the interval in $F_d/F_d$}
For an integer $m\geq1$, define
\begin{equation}\label{eq:automatic-states}
  Z_m=\phi_{10}(d^mJ)=1-d+d^{m+1}J
\end{equation}
and, for $1\leq k\leq m$, consider the intervals
\begin{align}\label{added-interval}
  d^kJ=\phi_{00}(d^{k-1}J),\quad Z_m,\quad dZ_m=\phi_{00}(Z_m),\quad\text{and}\quad d^2Z_m=\phi_{00}(dZ_m).
\end{align}
Consequently, if $ L(d)\leq t\leq\min\{U(d),V(d)\}$, then every
interval in (\ref{added-interval}) is contained in $R_t$ by Lemma \ref{lem:findingcover}.

The point that will produce a quotient is
\begin{equation}\label{eq:r0}
  r_0(d,t)=\frac{1-(1-2d)t}{d}.
\end{equation}

\begin{lem}\label{lem:rangefort-1}
For $m\geq1$, define
\begin{align}\label{eq:AmBm}
 A_m(d)=\frac{1-d^3+d^4-d^{m+4}}{1-2d+2d^{m+4}}\quad\text{and}\quad
 B_m(d)&=\frac{1-d^3+d^4-d^{m+4}}{1-2d}.
\end{align}
Then $ r_0(d,t)\in d^2Z_m$ if and only if $A_m(d)\leq t\leq B_m(d).$
Moreover,
\begin{equation}\label{eq:AmBm-limit}
  A_m(d)<B_m(d),\quad\text{and}\quad
  A_m(d),B_m(d)\nearrow
  B_\infty(d)=\frac{1-d^3+d^4}{1-2d}.
\end{equation}
\end{lem}

\begin{proof}
The left and right endpoints of $d^2Z_m$ are
\[
  d^2(1-d+d^{m+1})
  \quad\text{and}\quad
  d^2(1-d+d^{m+1}(1+2t)),
\]
respectively. The remaining assertions follows from (\ref{eq:r0}) and the straightforward computation.
\end{proof}

\begin{lem}\label{lem:target-to-quotient}
If $r_0(d,t)\in R_t$, then $t\in F_d/F_d$.
\end{lem}

\begin{proof}
Write $r_0=x+ty$ with $x,y\in\Gamma_d$.  Equation \eqref{eq:r0} gives $ 1-dx=t(1-2d+dy)$.
Remark that $1-dx=1-2d+d(2-x)$ and $1-2d+dy$
belong to
 $$1-2d+d\Gamma_d=1-2d+d(2-\Gamma_d)=(1-d)+dF_d\subset F_d.$$ 
Since $1-2d+d\Gamma_d\subset(0,\infty)$ for all $d<1/2$, then $t\in F_d/F_d$ as desired.
\end{proof}

\subsubsection{Finding the parameters--routine computations}
We now check the parameter ranges.  Define $\delta_1=0.1551428439174714607\ldots$ to be the unique zero in $(1/9,1/6)$ of
\begin{equation}\label{eq:q1}
  q_1(d)=d^5+d^4-d^3+3d^2+6d-1,
\end{equation}
and $\delta_2=0.1552404121503166368\ldots$ to be the unique zero there of
\begin{equation}\label{eq:q2}
  q_2(d)=d^3-3d^2-6d+1.
\end{equation}
The derivatives show that $q_1$ is increasing and $q_2$ is decreasing on
$[1/9,1/6]$.

Let $d_0=(7-\sqrt{33})/8$ be one of the zeros of the polynomial $-1+7d-4d^2.$
  Direct substitution at finite decimals gives
\begin{equation}\label{eq:threshold-order}
\begin{split}
 0.1551249&<\delta<0.1551250<0.1551428<\delta_1<0.1551429\\
 &<0.1552404<\delta_2<0.1552405<d_0<0.15693.
\end{split}
\end{equation}

We need the following elementary inequalities.

\begin{lem}\label{lem:parameter}
For $\delta\leq d\leq d_0$, $ L(d)<U(d)<V(d)$.
Besides, the following statements hold:
\begin{enumerate}
\item If $\delta<d\leq\delta_1$, then
      $L(d)<B_\infty(d)<U(d)$.
\item If $\delta_1<d\leq\delta_2$, then
      $B_1(d)>L(d)$ and $A_1(d)<U(d)$.
\item If $\delta_2<d\leq d_0$, set
\begin{equation}\label{eq:CD}
  {C}(d)=\frac{1-d^3}{1-2d+2d^3}
  \quad\text{and}\quad D(d)=\frac{1-d^3}{1-2d}.
\end{equation}
Then ${C}(d)<L(d)<D(d)$.
\end{enumerate}
\end{lem}

\begin{proof}
All denominators below are positive.  Direct subtraction gives that
\begin{align}
 U(d)-L(d)&=\frac{d^2\cdot(-1+8d-7d^2-4d^3+4d^4)}{(1+d)(1+d-4d^3)},\label{eq:UL}\\
 V(d)-U(d)&=\frac{-1+8d-5d^2-12d^3+10d^4}{(1-3d)(1+d-4d^3)}.\label{eq:VU}
\end{align}
Let $c(d)=-1+8d-7d^2-4d^3+4d^4$ and $e(d)=-1+8d-5d^2-12d^3+10d^4$.
Since $\delta>3/20$ and $ c(3/20),e(3/20)>0$
and $c',e'>16/3$ on $[3/20,1/6]$, equations
\eqref{eq:UL}--\eqref{eq:VU} prove that $ L(d)<U(d)<V(d)$.

Next, since $P(d)$ is decreasing on $[1/9,1/6]$ and
\begin{equation}\label{eq:Binf-minus-L}
 B_\infty(d)-L(d)=-\frac{P(d)}{(1-2d)(1+d)}.
\end{equation}
Thus, $B_\infty(d)>L(d)$ when $d>\delta$.  Furthermore,
\begin{equation}\label{eq:U-minus-Binf}
 U(d)-B_\infty(d)=
 \frac{h(d)}{(1+d-4d^3)(1-2d)},
\end{equation}
where
\[
 h(d)=1-7d+2d^2+11d^3-4d^4-d^5-4d^6+4d^7.
\]
The polynomial $h$ is decreasing on $[0,1/6]$, and
$h(621/4000)>0$.  Since $\delta_1<621/4000$, this proves part (1).

For $m=1$, subtraction gives
\begin{equation}\label{eq:B1-minus-L}
 B_1(d)-L(d)=\frac{(1-d)q_1(d)}{(1-2d)(1+d)}\quad\text{and}\quad  U(d)-A_1(d)=
 \frac{h_1(d)}{(1+d-4d^3)(1-2d+2d^5)},
\end{equation}
where
\[
 h_1(d)=1-7d+2d^2+11d^3-4d^4+4d^5-7d^6.
\]
The polynomial $h_1$ is decreasing on $[0,1/6]$, and
$h_1(621/4000)>0$.  Since $\delta_2<621/4000$, therefore part (2) follows.

For part (3), similarly, we have
\begin{equation}\label{eq:D-minus-L}
 D(d)-L(d)=\frac{(d-1)q_2(d)}{(1-2d)(1+d)}\quad\text{and}\quad L(d)-C(d)=\frac{f(d)}{(1+d)(1-2d+2d^3)}
\end{equation}
where
\[
 f(d)=1-7d+3d^2+8d^3-5d^4-2d^5+2d^6.
\]
The polynomial $f$ is decreasing on $[0,1/6]$, and
$f(157/1000)>0$.  Since $d_0<157/1000$, part (3) follows and the proof is complete.
\end{proof}

\begin{prop}\label{prop:delta-to-d0}
If $\delta<d\leq d_0$, then $F_d/F_d$ contains a nondegenerate interval.
\end{prop}

\begin{proof}
Suppose that $\delta<d\leq\delta_1$.  By
Lemma \ref{lem:parameter}, $ L(d)<B_\infty(d)<U(d)<V(d).$ Choose a sufficiently large finite $m$ in \eqref{eq:AmBm-limit}.  Then $$ \max\{L,A_m\}<\min\{U,B_m\}.$$
For each $\max\{L,A_m\}<t<\min\{U,B_m\}$, we have $A_m(d)\leq t\leq B_m(d)$, which implies that, from Lemma \ref{lem:rangefort-1}, 
$r_0(d,t)\in R_t$.  Lemma \ref{lem:target-to-quotient} gives
$t\in F_d/F_d$ for $\delta<d\leq\delta_1$.

Assume $\delta_1<d\leq\delta_2$, using $m=1$, we obtain, from Lemma \ref{lem:parameter}, 
$B_1>L$ and $A_1<U<V$.  Thus
\[
  \max\{L,A_1\}<\min\{U,B_1\}.
\]
The same argument gives an interval in $F_d/F_d$. 

It remains to consider $\delta_2<d\leq d_0$.  A direct calculation shows that $ r_0(d,t)\in d^2J$ if and only if $C(d)\leq t\leq D(d)$,
with $C,D$ as in \eqref{eq:CD}.  Lemma \ref{lem:parameter} gives
\[
  C<L<\min\{D,U,V\}.
\]
Hence, every $t$ in the nondegenerate interval
$\bigl(L,\min\{D,U,V\}\bigr)$ satisfies $ L(d)\leq t\leq\min\{U(d),V(d)\}$, and therefore $J\subset R_t$. Since
$d^2J\subset R_t$, Lemma \ref{lem:target-to-quotient} applies and the proof is complete.
\end{proof}

\section{Proof of Theorem \ref{mainthm}--part \rm{II}}\label{sec:partII}

The rest of the parameter range has a shorter description.  Define
\begin{equation}\label{eq:Gd}
  G_d=\left\{1+\sum_{n=1}^{\infty}\varepsilon_nd^n:
  \varepsilon_n\in\{-1,0,1\}\right\}.
\end{equation}
The set $(1-d)G_d$ is the subset of $F_d$ whose first digit is $1$.
Consequently,
\begin{equation}\label{eq:G-quotient}
  G_d/G_d\subset F_d/F_d.
\end{equation}

Fix $1<t<2$.  The equation $t\in G_d/G_d$ is equivalent to
\begin{equation}\label{eq:ratio-expansion}
  t-1=\sum_{n=1}^{\infty}(a_n-tb_n)d^n,
  \quad a_n,b_n\in\{-1,0,1\}.
\end{equation}
We first use only the seven digits
\begin{equation}\label{eq:seven-digits}
  \mathcal B_t=\{-t,-1,1-t,0,t-1,1,t\}.
\end{equation}

\begin{lem}\label{lem:seven-digit-interval}
Let
\begin{equation}\label{eq:ell-u}
  \ell(d)=\frac{2(1-d)}{1+d},
  \quad\text{and}\quad u(d)=\frac{1-d}{1-3d}.
\end{equation}
If $1<t<2$ and $\ell(d)\leq t\leq u(d)$, then every point of
\begin{equation}\label{eq:Jt}
  J_t=\left[-\frac{dt}{1-d},\frac{dt}{1-d}\right]
\end{equation}
has an expansion $z=\sum_{n=1}^{\infty}c_nd^n$, where $c_n\in\mathcal B_t$.
\end{lem}

\begin{proof}
    It suffices to show that \begin{align}\label{coverforJt}
J_t\subset\bigcup_{c\in  \mathcal B_t}d\left (c+J_t\right).
    \end{align}
   Any consecutive intervals on the right hand side of (\ref{coverforJt}) overlaps if the right endpoint of the former interval is larger than the left endpoint of the later interval. The assumptions $1<t<2$ and $\ell(d)\leq t\leq u(d)$ guarantee the overlapping property of the consecutive intervals. 
\end{proof}

We now allow the omitted extreme digit $1+t$ for finitely many places.
For $m\geq1$, let
\begin{equation}\label{eq:ambm}
  a_m(d)=\frac{1-d^{m+1}}{1-2d+2d^{m+1}},
  \quad
  b_m(d)=\frac{1-d^{m+1}}{1-2d}\quad\text{and}\quad b_\infty(d)=\frac{1}{1-2d}.
\end{equation}
The number $b_\infty(d)$ is the limit of both $a_m(d)$ and $b_m(d)$.

\begin{lem}\label{lem:moreseven}
Suppose that $\ell(d)\leq t\leq u(d)$ and
$a_m(d)\leq t\leq b_m(d)$ for some $m\geq1$.  Then
$t\in G_d/G_d$.
\end{lem}

\begin{proof}
We need to establish \eqref{eq:ratio-expansion}. By Lemma \ref{lem:seven-digit-interval}, if we take the digit $1+t$ in the first $m$
places, then the remaining digits can be
chosen from $\mathcal B_t$ if and only if this number lies in $d^mJ_t$, which is equivalent to
\begin{align*}
     \left|t-1-(1+t)\frac{d(1-d^m)}{1-d}\right|
  \leq\frac{d^{m+1}t}{1-d}.
\end{align*}
Multiplication by $1-d$ and rearrangement give 
$a_m(d)\leq t\leq b_m(d)$.  As a result, we have
$$
 t-1=\sum^m_{n=1}(1+t)d^n+\sum_{n=m+1}^{\infty}(a_n-tb_n)d^n,
$$
where $a_n,b_n\in\{-1,0,1\}$, and hence $t\in G_d/G_d$.
\end{proof}

\begin{prop}\label{prop:d0-to-quarter}
If $d_0<d<1/4$, then $F_d/F_d$ contains a nondegenerate interval. In particular, if $d=1/6$, then $[10/7,35/24]\subset F_d/F_d$.
\end{prop}

\begin{proof}
A direct computation shows that $$
 b_\infty(d)-\ell(d)
 =\frac{-1+7d-4d^2}{(1-2d)(1+d)}>0,\quad\forall d_0<d<1/4.
$$
Besides, $u(d)-b_\infty(d)>0$ and 
$1<b_\infty<2$ for $0<d<1/4$.  Thus, if $d_0<d<1/4$, then
\[
  1<\ell(d)<b_\infty(d)<\min\{u(d),2\}.
\]
Since $a_m,b_m\nearrow b_\infty$, then, for all sufficiently large $m$, the intersection
\[[\ell(d),u(d)]\cap[a_m(d),b_m(d)]\cap(1,2)
\]
has nonempty interior.  Lemma \ref{lem:moreseven}, together with
\eqref{eq:G-quotient}, reveal that $F_d/F_d$ contains a nondegenerate interval. In particular, if $d=1/6$ and $m=1$, then
\[
 \ell=\frac{10}{7},\quad u=\frac53,
 \quad a_1=\frac{35}{26},\quad \text{and}\quad b_1=\frac{35}{24}.
\]
Their intersection is $[10/7,35/24]$, which implies that $[10/7,35/24]\subset F_d/F_d$.
\end{proof}

\subsection{Proof of Theorem \ref{mainthm} and Corollary \ref{Mattila}}
\begin{proof}
Propositions \ref{prop:delta-to-d0} and \ref{prop:d0-to-quarter} cover
$\delta<d<1/4$.  The explicit assertion is contained in Proposition
\ref{prop:d0-to-quarter}. This complete the proof of Theorem \ref{mainthm}. Corollary \ref{Mattila} can be obtained by combining Theorem \ref{mainthm} and Proposition \ref{prop:criterion}.
\end{proof}

\bigskip
\noindent{\bf AI Disclosure:}
Generative AI tools were used for language editing, proofreading and numerical computing. The choice of the parameter $r_0(d,t)$ is also suggested by AI.  All other mathematical results, arguments, proofs, and conclusions were developed, verified, and approved by the authors.

\end{document}